\documentclass{article}
\usepackage{graphicx} 

\usepackage{graphicx} 
\usepackage{amssymb}
\usepackage{amscd}
\usepackage{amsthm}
\usepackage{amsmath}
\usepackage{xcolor}
\usepackage{hyperref}
\usepackage{cleveref}
\usepackage{tikz-cd}
\usepackage{quiver}
\usepackage{todonotes}

\newcommand{\xra}{\xrightarrow}

\newtheorem{nul}{}[subsection]
\newtheorem{thm}[nul]{Theorem}
\newtheorem{prop}[nul]{Proposition}
\newtheorem{cor}[nul]{Corollary}
\newtheorem{lemma}[nul]{Lemma}

\newtheorem{defn}[nul]{Definition}
\newtheorem*{main-theorem}{Main Theorem}

\theoremstyle{remark}

\newtheorem{ntn}[nul]{Notation}

\theoremstyle{definition}
\newtheorem*{ack}{Acknowledgements}

\newcommand{\Z}{\mathbb{Z}}

\newcommand{\R}{\mathbb{R}}
\newcommand{\Q}{\mathbb{Q}}
\newcommand{\C}{\mathbb{C}}

\newcommand{\uZ}{\underline{\mathbb{Z}}}
\newcommand{\res}{\mathrm{res}}
\newcommand{\tr}{\mathrm{tr}}
\newcommand{\id}{\mathrm{id}}
\newcommand{\cS}{\mathcal{S}}
\newcommand{\Sp}{\mathrm{Sp}}
\newcommand{\mS}{\mathbb{S}}

\title{On the rational $C_2$-homotopy type of ${BSU_{\R}}_m$}
\author{Eunice Sukarto}
\date{}

\begin{document}

\maketitle

\begin{abstract}
    Motivated by a problem in motivic homotopy theory considered by Asok-Fasel-Hopkins, we give a description of the rational $C_2$-equivariant homotopy type of the classifying space ${BSU_{\mathbb{R}}}_m$ in terms of equivariant Eilenberg-Maclane spaces.
\end{abstract}

\section{Introduction}
In \cite{afh22}, Asok, Fasel, and Hopkins study the motivic classifying space $BSL_m$ and show that it has a 'well-behaved' rationalization, either if $m$ is odd or if $m$ is even and -1 is a sum of squares in the base field. In particular, they show that if either of these conditions is satisfied, the Chern classes $c_i:BSL_m\to K(\Z(i),2i)$ together give a rational $\mathbb{A}^1$-weak equivalence
\[BSL_m\to \prod_{i=2}^m K(\Z(i),2i)\]
where $\Z(i)$ is Voevodsky's motivic complex (see \cite[Def 3.1]{mvw06}).

This manifests itself in ordinary homotopy theory as the fact that over the complex numbers, $BSL_m$ is equivalent to $BSU_m$ which is rationally a product of Eilenberg-Maclane spaces 
\[BSU_m\to \prod_{i=1}^m K(\Z, 2i).\] 
However, over the reals, $BSL_m$ is equivalent to $BSO_m$ where  
\[BSO_m\to \prod_{i=1}^m K(\Z,2i)\] is only a rational equivalence when $m$ is odd. If $m$ is even, there is an  Euler class living in degree $m$ which is not accounted for. To better understand this phenomena, we consider these two cases together, which leads to $C_2$-equivariant homotopy theory.

The group $C_2$ acts on $BSU_m$ by complex conjugation, making it into a $C_2$-equivariant space ${BSU_{\R}}_m$ classifying Real vector bundles in the sense of Atiyah. It has underlying space $BSU_m$ and fixed points $BSO_m$.
In this paper, we will give a description of the rational $C_2$-equivariant homotopy type of ${BSU_{\R}}_m$ in terms of equivariant Eilenberg-Maclane spaces.

If $m=2n+1$ is odd, it is known that the equivariant Chern classes give a rational $C_2$-equivalence
\[{BSU_{\R}}_{(2n+1)}\to \prod_{i=2}^{2n+1} K(\uZ,i\rho),\]
where $\rho$ is the regular representation and $K(\uZ,i\rho) = \Omega_{C_2}^{\infty} (\mS^{i\rho}\otimes \uZ)$ is the equivariant Eilenberg-Maclane space corresponding to the suspension of the constant Mackey functor $\uZ$ by the representation sphere $\mS^{i\rho}$. We consider the case where $m=2n$ is even. Let $\iota^2, N:K(A,2n\rho)\to K(A,4n\rho)$ denote the squaring and norm maps, respectively. Then $\iota^2-N$ factors through $K(I,4n)$ where $I$ is the augmentation ideal Mackey functor. Our main theorem is the following.

\begin{thm}[Theorem \ref{thm:main}]
    The map
    \[
        {BSU_{\R}}_{2n}\to \left(\prod_{i=2}^{2n-1}K(\uZ,i\rho)\right) \times \mathrm{fib} \left(K(A,2n\rho) \xra{\iota^2-N} K(I,4n)\right)
    \]
    given by the Chern classes $c_i$ for $i<n$ and the Euler class is a rational $C_2$-equivalence.    
\end{thm}
The fact that the Euler class maps into the fiber of $\iota^2-N$ is the condition that the norm of the Euler class is equal to its square: $e(NV) = e(V\otimes \C) = e(V)^2$ for $V$ an even dimensional real vector bundle.

\begin{ack}
    I would like to thank my advisor Mike Hopkins for suggesting this problem, for generously sharing his ideas, and for all his guidance throughout this project and over the years.
\end{ack}

\begin{ntn}
    \begin{enumerate}
        \item All spaces and spectra will be rationalized. For a space or spectrum $X$, the notation $X$ will always refer to its rationalization.

        \item Let $\mathcal{S}_G$ and $\mathrm{Sp}_G$ denote the categories of pointed $G$-spaces and $G$-spectra.

        \item For $V$ a $G$-representation, $S^V$ and $\mS^V$ denote the representation sphere in $\cS_G$ and $\Sp_G$, respectively.
        
        \item Let $G=C_2$ denote the cyclic group of order 2 and $e$ the trivial subgroup. Let $\sigma$ and $\rho$ denote the sign and regular representations of $C_2$.

        \item We will not distinguish between a Mackey functor and its corresponding Eilenberg-Maclane $G$-spectrum.
        
        \item For $M$ a Mackey functor and $V$ a $G$-representation, we write $K(M,V) = \Omega_G^{\infty} (\mS^V\otimes M)$ for the infinite $G$-loop space of the $G$-spectrum $\mS^V\otimes M$.
    \end{enumerate}
\end{ntn}

\section{Preliminaries}
$C_2$ acts on ${BSU_{\R}}_m$ by complex conjugation with fixed points $BSU_m^{C_2}\simeq BSO_m$. The rational (non-equivariant) cohomology of $BSU_m$ and $BSO_m$ are polynomials in characteristic classes
\begin{align}
    H^*(BSU_m; \Q) &\cong \Q[c_2,\cdots, c_m] \nonumber\\
    H^*(BSO_{2n}; \Q) &\cong \Q[p_1,\cdots, p_{n-1}, e_{2n}] \label{eq:q-coh}\\
    H^*(BSO_{2n+1}; \Q) &\cong \Q[p_1,\cdots, p_n] \nonumber
\end{align}
where $c_i$, $p_i$, and $e_{2n}$ are the Chern, Pontryagin, and Euler classes with $|c_i|=2i$, $|p_i|=4i$, and $|e_{2n}| = 2n$. The inclusion of fixed points $BSO_m\to BSU_m$ is given by the complexification of vector bundles. In terms of cohomology, the map $BSO_{2n+1}\to BSU_{2n+1}$ sends $c_{2i}$ to $p_i$, and the map $BSO_{2n}\to BSU_{2n}$ sends $c_{2i}$ to $p_i$ for $i<n$ and $c_{2n}$ to $e_{2n}^2$.

\subsection{Mackey functors}
We follow the notation in \cite{dug05}.
The orbit category $\mathcal{O}_{C_2}$ is given by
\[\begin{tikzcd}
	{C_2/e} & {C_2/C_2}
	\arrow["{t}", from=1-1, to=1-1, loop, in=145, out=215, distance=10mm]
	\arrow["i", shift left, from=1-1, to=1-2].
\end{tikzcd}\]
where $it=i$ and $t^2 = \id$. A Mackey functor $M$ for $C_2$ consists of abelian groups $M(C_2/e)$ and $M(C_2/C_2)$ with a $C_2$-action on $M(C_2/e)$, together with restriction and transfer maps
\[\res_M:M(C_2/e)\leftrightarrows M(C_2/C_2):\tr_M\]
satisfying the conditions in \cite[2.4]{dug05}. 

The constant Mackey functor $\uZ$ is given by $\uZ(G/H)=\Q$ where all restriction maps are identities. For $C_2$, this is given by $\uZ(C_2/e) = \uZ(C_2/C_2) = \Q$ with $C_2$ acting trivially on $\uZ(C_2/e)$. The transfer is multiplication by 2.

The Burnside Mackey functor $A$ is the Mackey functor where $A(G/H)$ is the Grothendieck group of finite $H$-sets, with restriction and transfer given by restriction and induction. For $C_2$, $A(C_2/e) = \Q$ with trivial $C_2$-action and $A(C_2/C_2) = \Q[T]/T^2-2T$, where $T$ corresponds to the $C_2$-set $C_2/e$. The restriction and transfer maps are given by $a+bT\mapsto a+2b$ and $1\mapsto T$, respectively.

The augmentation ideal $I$ is the Mackey functor with $I(G/H)$ given by the kernel of the augmentation map $A(G/H)\to A(G/e)=\Q$ which takes an $H$-set to its cardinality (\cite[Def 4.5]{hhr11}). For $C_2$, $I(C_2/e)=0$ and $I(C_2/C_2)\cong \Q\{T-2\}\cong \Q$ is the free $\Q$-module on the generator $T-2$. The constant Mackey functor $\uZ$ is the quotient $A/I$.

If $M$ is a Mackey functor and $S$ a $G$-set, the tensor product $M\otimes S_+$  is the Mackey functor defined by $(M\otimes S_+)(B) = M(B\times S)$. The permutation Mackey functor $\uZ[C_2]:= \uZ\otimes {C_2}_+$ is given by $\uZ[C_2](C_2/e)\cong \Q[C_2]\cong \Q[\tau]/\tau^2-1$ with $C_2$ acting by $\tau$, and $\uZ[C_2](C_2/C_2)=\Q\{1+\tau\}\cong \Q$. The restriction map is $1\mapsto 1+\tau$ and the transfer is the sum of coefficients $a+b\tau\mapsto (a+b)(1+\tau)$.

\subsection{Equivariant characteristic classes}
The $C_2$-CW structure of $MU_{\R}$ gives a splitting 
\[\uZ\wedge MU_{\R} \simeq \uZ\wedge \left(\bigvee_d \mS^{\frac{|d|}{2}\rho} \right)\] 
where $d$ ranges over the $C_2$-cells of $MU_{\R}$, which are the monic monomials in $MU_*\cong \Z[x_1,x_2,\cdots]$ (\cite[Section 5]{hhr11}). Since $MU_{\R}$ is the Thom spectrum of $BU_{\R}$, the equivariant Chern class $c_i$ which corresponds to $x_i$, naturally lives in $\mS^{i\rho} \otimes \uZ$. $\mS^{i\rho} \otimes \uZ$ is computed by the complex $\Z(i)[2i]$ given by
\begin{equation}\label{eq:complex}
    \uZ[C_2]\xra{1-(-1)^i\tau} \cdots \uZ[C_2]\xra{1+\tau}\uZ[C_2]\xra{1-\tau}\uZ[C_2] \to \uZ
\end{equation}
with $\uZ$ in degree $i$ and the leftmost $\uZ[C_2]$ in degree $2i$. Here $[-]$ denotes the shift/suspension functor.
The map $\uZ[C_2]\to \uZ$ is induced by the transfer in $\uZ[C_2]$. On underlying, $\mS^{i\rho} \otimes \uZ$ is given by
\[\Q[C_2]\xra{1-(-1)^i\tau} \cdots \Q[C_2]\xra{1+\tau}\Q[C_2]\xra{1-\tau}\Q[C_2] \to \Q\]
where the last map is the sum of coefficients. On fixed points, it is given by
\[\Q\to \cdots \to \Q\xra{\cdot 2} \Q\xra{0} \Q\xra{\cdot 2}\Q\]
where the leftmost map is $\cdot 2$ if $i$ is odd and 0 if $i$ is even. Thus, $K(\uZ, i\rho)$ is a $C_2$-space with underlying $K(\Q,2i)$ and fixed points $K(\Q,2i)$ if $i$ is even and trivial if $i$ is odd.

The Chern classes together give a map
\[{BSU_{\R}}_{m}\to \prod_{i=2}^{m} K(\uZ,i\rho).\]
By (\ref{eq:q-coh}), this is a rational $C_2$-equivalence if $m$ is odd.
However, if $m=2n$ is even, the Euler class $e_{2n}\in H^{2n}(BSO_{2n};\Q)$ is not in the image of this map. We would like a description of ${BSU_{\R}}_{2n}$ which captures this Euler class.

The Thom class gives a map ${MU_{\R}}_m\to \uZ$. Since the cell structure of ${MSU_{\R}}_m$ gives no obstruction to lifting this to a map ${MSU_{\R}}_m\to A$, the equivariant Euler class of ${BSU_{\R}}_m$ naturally lives in $\mS^{m\rho}\otimes A$. Since $A\otimes {C_2}_+ \simeq \uZ\otimes {C_2}_+$, $K(A,m\rho)$ has the same underlying space as $K(\uZ, m\rho)$ but differs by an additional $\Q$ in degree $m$ on fixed points. The bottom part of the corresponding complex (\ref{eq:complex}) is $\uZ[C_2]\to A$, which is $1+\tau: \Q[C_2]\to \Q$ on underlying and $T:\Q\to \frac{\Q[T]}{T^2-2T}$ on fixed points. If $n$ is even, this Euler class detects the Euler class $e_{2n}\in H^{2n}(BSO_{2n};\Q)$. 

\section{The squaring and the norm maps}
Let $K_m$ denote the Eilenberg-Maclane space $K(\Q,m)$. Let $\iota^2, N:K(A,2n\rho)\to K(A,4n\rho)$ denote the squaring and norm maps, respectively. In this section, we will identify $\iota^2$ and $N$ on underlying spaces and fixed points. 
\begin{prop}
    Let $x$ and $y$ be the fundamental classes in $K_{4n}$ and $K_{8n}$, respectively.
    The squaring map $\iota^2:K(A,2n\rho)\to K(A,4n\rho)$ is given on underlying spaces by $K_{4n}\to K_{8n}$, $y\mapsto y^2$ and on fixed points by $K_{2n}\times K_{4n}\to K_{4n}\times K_{8n}$, $(x,y)\mapsto (x^2,y^2)$.

    The norm map $N:K(A,2n\rho)\to K(A,4n\rho)$ is given on underlying spaces by $K_{4n}\to K_{8n}$, $y\mapsto y^2$ and on fixed points by $K_{2n}\times K_{4n}\to K_{4n}\times K_{8n}$, $(x,y)\mapsto (y,y^2)$.
\end{prop}

\begin{lemma}
    The inclusion of fixed points $j_{\Z}:K(\uZ, 2n\rho)^{C_2}\to K(\uZ, 2n\rho)$ is the identity map on $K_{4n}$. 
    The inclusion of fixed points $j_A:K(A,2n\rho)^{C_2}\to K(A,2n\rho)$ is the map $K_{2n}\times K_{4n}\to K_{4n}$ given by projection onto the second factor.
\end{lemma}

\begin{proof}
    The first statement is clear. For the second, we need to identify $j_A$ on each component of $K(A,2n\rho)^{C_2}$. $j_A$ is equal to the map on fixed points of the infinite $C_2$-loop map $K(A, 2n\rho)\to K(\uZ, 2n\rho)$ induced by the projection $A\to \uZ$. Since the inclusion of fixed points $j_{\Z}$ on $K(A,2n\rho)$ is the identity map, $j_A$ is the identity on $K_{4n}$. Since there are no nonzero maps of spectra $\Sigma^{2n}H\Q\to \Sigma^{4n}H\Q$, $j_A$ is 0 on $K_{2n}$.
\end{proof}

\subsection{The squaring map}
The squaring map $\iota^2$ is given as the composite of the diagonal $\Delta$ followed by the cup product
\[\begin{tikzcd}
	{\iota^2:} & {K(A,2n\rho)} & {K(A,2n\rho)\times K(A,2n\rho)} & {K(A,4n\rho)} \\
	{\mathrm{u}} & {K_{4n}} & {K_{4n}\times K_{4n}} & {K_{8n}} \\
	{\mathrm{fp}} & {K_{2n}\times K_{4n}} & {(K_{2n}\times K_{4n})\times (K_{2n}\times K_{4n})} & {K_{4n}\times K_{8n}}
	\arrow["\Delta", from=1-2, to=1-3]
	\arrow["\smile", from=1-3, to=1-4]
	\arrow["\Delta", from=2-2, to=2-3]
	\arrow["\smile", from=2-3, to=2-4]
	\arrow["{j_A}", from=3-2, to=2-2]
	\arrow["\Delta", from=3-2, to=3-3]
	\arrow["{j_A\times j_A}"', from=3-3, to=2-3]
	\arrow[from=3-3, to=3-4]
	\arrow["{j_A}", from=3-4, to=2-4]
\end{tikzcd}\]
where u denotes the underlying space and fp the fixed points. $C_2$ acts by the diagonal action on $K(A,2n\rho)\times K(A,2n\rho)$. $\iota^2$ is the squaring map on underlying spaces. On fixed points, a diagram chase shows that $\iota^2$ is the squaring map on the component $K_{4n}$. We claim that it is also the squaring map on the component $K_{2n}$.

\begin{lemma}\label{lem:i}
    $\iota^2$ is given by the squaring map on the component $K_{2n}$ of $K(A,2n\rho)^{C_2}$.
\end{lemma}

For a $G$-representation $V$, the Euler class $a_V\in \pi_{-V}^G(\mS^0)$ is the element corresponding to the map $\mS^0\hookrightarrow \mS^V$ induced by the inclusion $0\subseteq V$ (\cite[Def 3.11]{hhr16}). For a Mackey functor $M$, $a_V:\mS^0\otimes M\to \mS^V\otimes M$ gives a natural transformation of $RO(G)$-graded cohomology theories $H_G^{\bigstar}(-;M)\xra{\cdot a_V} H_G^{\bigstar +V}(-;M)$ which is multiplication by $a_V$.

Let $a_{2n\sigma}: \mS^{2n}\to \mS^{2n\rho}$ be the (suspended) Euler class induced by the inclusion $0\subset 2n\sigma$. This induces a map of $C_2$-spectra $\mS^{2n}\otimes A\to \mS^{2n\rho}\otimes A$ hence an infinite $C_2$-loop map $K(A,2n)\to K(A,2n\rho)$.

\begin{lemma}\label{lem:euler-a}
    The Euler class $a_{2n\sigma}:K(A,2n)\to K(A,2n\rho)$ has image $K(I,2n\rho)$. In particular, $a_{2n\sigma}$ picks out the component $K_{2n}$ in $K(A,2n\rho)^{C_2}$.
\end{lemma}

\begin{proof}
    The map $\mS^{2n}\otimes A\to \mS^{2n\rho}\otimes A$ is given on underlying u and fixed points fp as follows.
    \[\begin{tikzcd}
	{\mathrm{u}} & 0 & {\Q[C_2]} && {\mathrm{fp}} & 0 & \Q \\
	& \vdots & \vdots &&& \vdots & \vdots \\
	& 0 & {\Q[C_2]} &&& 0 & \Q \\
	& \Q & \Q &&& {\frac{\Q[T]}{T^2-2T}} & {\frac{\Q[T]}{T^2-2T}}
	\arrow[from=1-2, to=1-3]
	\arrow[from=1-2, to=2-2]
	\arrow["{1-\tau}", from=1-3, to=2-3]
	\arrow[from=1-6, to=1-7]
	\arrow[from=1-6, to=2-6]
	\arrow[from=1-7, to=2-7]
	\arrow[from=2-2, to=3-2]
	\arrow[from=2-3, to=3-3]
	\arrow[from=2-6, to=3-6]
	\arrow[from=2-7, to=3-7]
	\arrow[from=3-2, to=3-3]
	\arrow[from=3-2, to=4-2]
	\arrow["{1+\tau}", from=3-3, to=4-3]
	\arrow[from=3-6, to=3-7]
	\arrow[from=3-6, to=4-6]
	\arrow["T", from=3-7, to=4-7]
	\arrow["\id", from=4-2, to=4-3]
	\arrow["\id", from=4-6, to=4-7]
    \end{tikzcd}\]
    The bottom map on fixed points is the identity since $a_{2n\sigma}$ is the inclusion of the bottom cell. 

    Looking at what it does on homotopy groups shows that the map $K(A,2n)\to K(A,2n\rho)$ is given by
    \[\begin{tikzcd}
	& {K(A,2n)} & {K(A,2n\rho)} \\
	{\mathrm{u}} & {K(\Q,2n)} & {K(\Q,4n)} \\
	{\mathrm{fp}} & {K(\frac{\Q[T]}{T^2-2T},2n)} & {K(\Q,2n)\times K(\Q,4n)}
	\arrow[from=1-2, to=1-3]
	\arrow["0", from=2-2, to=2-3]
	\arrow["{\res_A}", from=3-2, to=2-2]
	\arrow["{(\mathrm{proj},0)}"', from=3-2, to=3-3]
	\arrow["{j_A}", from=3-3, to=2-3]
    \end{tikzcd}\]
    where proj is the projection $\frac{\Q[T]}{T^2-2T}\to \Q$ sending $T$ to 0. Thus, $a_{2n\sigma}$ picks out $K_{2n}$.
\end{proof}
    
\begin{proof}[Proof of Lemma \ref{lem:i}]
    Let $x$ denote the fundamental class in $K_{2n}$.
    Since $a_{2n\sigma}$ is a natural transformation $H_{C_2}^2(-;A)\to H_{C_2}^{2\rho}(-;A)$,
    \[(a_{2n\sigma}x)^2 = a_{2n\sigma}^2 x^2 = a_{4n\sigma} x^2,\]
    so there is a commutative diagram
    \[\begin{tikzcd}
	{K(A,2n)} & {K(A,4n)} \\
	{K(A,2n\rho)} & {K(A,4n\rho)}
	\arrow["{\iota^2}", from=1-1, to=1-2]
	\arrow["{a_{2n\sigma}}"', from=1-1, to=2-1]
	\arrow["{a_{4n\sigma}}", from=1-2, to=2-2]
	\arrow["{\iota^2}"', from=2-1, to=2-2].
    \end{tikzcd}\]
    Since $\iota^2: K(A,2n)\to K(A,4n)$ is the squaring map on each component of its fixed points, a diagram chase together with Lemma \ref{lem:euler-a} shows that $\iota^2$ is the squaring map on the component $K_{2n}$ of $K(A,2n\rho)^{C_2}$.
\end{proof}

\subsection{The norm map}
Let $N:=N_H^G$ denote the norm functor for a subgroup $H\leq G$. The norm $N:\cS_H\to \cS_G$ on spaces and $N:\Sp_H\to \Sp_G$ on spectra are strong symmetric monoidal and are related by the adjunction
\[\Sigma_G^{\infty}: \mathcal{S}_G\rightleftarrows \mathrm{Sp}_G: \Omega_G^{\infty}.\] 
The left adjoint $\Sigma_G^{\infty}$ is strong symmetric monoidal and commutes with the norm: $N\Sigma_G^{\infty} = \Sigma_G^{\infty}N$. However, the right adjoint $\Omega_G^{\infty}$ is only lax symmetric monoidal, so $N\Omega_G^{\infty}\neq \Omega_G^{\infty}N$, but there is a map from the former to the latter. If $R\in \Sp_G$ is an equivariant commutative algebra (\cite[2.3]{hhr16}), there is a map $NR\to R$. For $X\in \cS_G$, there is a map $X\to NX$ which factors through the product $\prod_{i\in G/H}X_i$ (\cite[2.3.3]{hhr16}).

For $G=C_2$, let $t$ be the generator in $C_2$. For $X\in \mathcal{S}_{C_2}$, the map $X\to NX$ is given by
\[X\to X\times X\to NX= X\wedge X\]
where $X\times X$ and $X\wedge X$ are considered as (pointed) $C_2$-spaces with $C_2$ acting by swapping the two factors i.e. $t (x,y) = (y,x)$ and the first map sends $x$ to $(x, tx)$. 

Let $V$ be a $C_2$-representation. The norm map $N:K(A, V)\to K(A,V\otimes \rho)$ is given by the composite
\begin{equation}\label{eq:norm-composite}
    K(A,V)\to K(A,V)\times K(A,V)\to NK(A,V)\to \Omega_G^{\infty} N(\mS^V\otimes A)\to K(A, V\otimes \rho)
\end{equation} 
where $NK(A,V)$ and $N(\mS^V\otimes A)$ are the space and spectrum level norms, respectively. The last map comes from the ring structure of $A$: $N(\mS^V\otimes A) \simeq \mS^{V\otimes \rho} \otimes NA\to \mS^{V\otimes \rho}\otimes A$. 

\begin{lemma}\label{lem:norm-u}
    The norm map $N:K(A,V)\to K(A,V\otimes \rho)$ is given on underlying spaces as the cup square $K(A,|V|)\to K(A,2|V|)$ where $|V| = \dim V$.
\end{lemma}

\begin{proof}
    On underlying, the composite (\ref{eq:norm-composite}) is given by $x\mapsto (x,tx)\mapsto x\cdot t(tx) = x^2$, which is the cup square.
\end{proof}

Let $\eta:\mS^0\to A$ denote the unit. Consider the inclusion of the bottom cell $f_V:S^V\to K(A,V)$ adjoint to $\mS^V\otimes \eta$.
\begin{lemma}\label{lem:norm-cell}
    The norm commutes with the inclusion of the bottom cell: $N\circ f_V = f_{V\otimes \rho} \circ N$.
\end{lemma}

\begin{proof}
    Since $f_V$ commutes with the first two maps in (\ref{eq:norm-composite}), it suffices to show that the diagram 
    \[\begin{tikzcd}
	{NK(A,V)} & {\Omega_G^{\infty}N(\mS^V\otimes A)} & {K(A,V\otimes \rho)} \\
	{NS^V=S^{V\otimes \rho}} && {S^{V\otimes \rho}}
	\arrow[from=1-1, to=1-2]
	\arrow[from=1-2, to=1-3]
	\arrow["{Nf_V}", from=2-1, to=1-1]
	\arrow["\id", from=2-1, to=2-3]
	\arrow["{f_{V\otimes \rho}}"', from=2-3, to=1-3]
    \end{tikzcd}\]
    commutes. 
    
    If $g:X\to \Omega_G^{\infty}Y$ is adjoint to $h:\Sigma_G^{\infty}X\to Y$, then the composite $NX\xra{Ng} N\Omega_G^{\infty}Y\to \Omega_G^{\infty}NY$ is adjoint to $\Sigma_G^{\infty}NX = N\Sigma_G^{\infty}X\xra{Nh} NY$. Thus, the clockwise composite is adjoint to 
    \[\Sigma_G^{\infty}NS^V = N\mS^V \xra{N(\mS^V\otimes \eta)} N(\mS^V\otimes A) = \mS^{V\otimes \rho}\otimes NA\to \mS^{V\otimes \rho} \otimes A,\]
    which is equal to $N\mS^V\otimes (\mS^0\xra{N\eta} NA\to A) = N\mS^V \otimes \eta$.
\end{proof}

We now apply this to $V=2n\rho$. The norm map factors as
\begin{equation}\label{diag:norm}
    \begin{tikzcd}
	{N:} & {K(A,2n\rho)} & {K(A,2n\rho)\times K(A,2n\rho)} & {K(A,4n\rho)} \\
	{\mathrm{u}} & {K_{4n}} & {K_{4n}\times K_{4n}} & {K_{8n}} \\
	{\mathrm{fp}} & {K_{2n}\times K_{4n}} & {K_{4n}} & {K_{4n}\times K_{8n}}
	\arrow[from=1-2, to=1-3]
	\arrow[from=1-3, to=1-4]
	\arrow[from=2-2, to=2-3]
	\arrow[from=2-3, to=2-4]
	\arrow["{j_A}", from=3-2, to=2-2]
	\arrow["{j_A}", from=3-2, to=3-3]
	\arrow["\Delta"', from=3-3, to=2-3]
	\arrow[from=3-3, to=3-4]
	\arrow["{j_A}", from=3-4, to=2-4]
    \end{tikzcd}
\end{equation}
where $C_2$ acts on $K(A,2n\rho)\times K(A,2n\rho)$ by swapping the factors. By Lemma \ref{lem:norm-u}, the map on underlying spaces is the cup square. By a diagram chase, the map on fixed points is given by 
\begin{equation}\label{eq:norm-fp}
    (x,y)\mapsto y\mapsto (?,y^2)
\end{equation} 
where $x$ and $y$ are the fundamental classes in $K_{2n}$ and $K_{4n}$, respectively. We claim that the map to the $K_{4n}$ component if the identity i.e. that $?=y$.

\begin{lemma}
    On fixed points, the norm map $N:K(A,2n\rho)\to K(A,4n\rho)$ is given by $(x,y)\mapsto (y,y^2)$.
\end{lemma}

\begin{proof}
    The map $K(A,2n\rho)\times K(A,2n\rho)\to K(A,4n\rho)$ in (\ref{diag:norm}) factors as the top row of the commutative diagram
    \[\begin{tikzcd}
	{K(A,2n\rho)\times K(A,2n\rho)} & {NK(A,2n\rho)} & {K(A,4n\rho)} \\
	{S^{2n\rho}\times S^{2n\rho}} & {NS^{2n\rho}=S^{4n\rho}} \\
	{S^{4n}} & {S^{4n}} & {K(A,4n)}
	\arrow[from=1-1, to=1-2]
	\arrow[from=1-2, to=1-3]
	\arrow[from=2-1, to=1-1]
	\arrow["{\mathrm{proj}}", from=2-1, to=2-2]
	\arrow[from=2-2, to=1-2]
	\arrow["{f_{4n\rho}}"', from=2-2, to=1-3]
	\arrow["\Delta", from=3-1, to=2-1]
	\arrow["\id", from=3-1, to=3-2]
	\arrow["{a_{4n\sigma}}", from=3-2, to=2-2]
	\arrow["{f_{4n}}", from=3-2, to=3-3]
	\arrow["{a_{4n\sigma}}"', from=3-3, to=1-3].
    \end{tikzcd}\]
    Since the diagonal $\Delta$ is the inclusion of fixed points, the left column factors through the fundamental class $y:S^{4n}\to K_{4n} = \left( K(A,2n\rho)\times K(A,2n\rho)\right)^{C_2}$. By Lemma \ref{lem:euler-a}, the Euler class $a_{4n\sigma}$ picks out the component $K_{4n}$ in $K(A,4n\rho)^{C_2}$. The commutativity of the outer diagram shows that $?=y$ in (\ref{eq:norm-fp}).
\end{proof}

\section{The fiber}
Consider the map $\iota^2-N:K(A,2n\rho)\to K(A,4n\rho)$. Since it is 0 on underlying spaces and is given by $K_{2n}\times K_{4n}\to K_{4n}$, $(x,y)\mapsto x^2-y$ on fixed points, $\iota^2-N$ factors through the inclusion $K(I,4n)=K(I,4n\rho)\to K(A,4n\rho)$. 
\begin{defn}
    Let $F_{2n}$ be the fiber of the map $\iota^2 -N: K(A,2n\rho)\to K(I,4n\rho)$. The inclusion $F_{2n}\to K(A,2n\rho)$ is given by
    \[\begin{tikzcd}
	& {F_{2n}} & {K(A,2n\rho)} \\
	{\mathrm{u}} & {K_{4n}} & {K_{4n}} \\
	{\mathrm{fp}} & {K_{2n}} & {K_{2n}\times K_{4n}}
	\arrow[from=1-2, to=1-3]
	\arrow["\id", from=2-2, to=2-3]
	\arrow["{x^2}", from=3-2, to=2-2]
	\arrow["{(\id, x^2)}"', from=3-2, to=3-3]
	\arrow["{j_A}"', from=3-3, to=2-3].
    \end{tikzcd}\]
\end{defn}

The equivariant Euler class $\epsilon\in {BSU_{\R}}_{2n}$ is classified by a map $\epsilon: {BSU_{\R}}_{2n}\to K(A,2n\rho)$ given on underlying by $e_{4n}=c_{2n}\in H^{4n}(BSU_{2n};\Q)$ and on fixed points by $(e_{2n}, e_{2n}^2)$ where $e_{2n}\in H^{2n}(BSO_{2n};\Q)$. Since $(\iota^2-N)(\epsilon) = 0$, $\epsilon$ factors through $F_{2n}$. By (\ref{eq:q-coh}), we have the following.

\begin{thm}\label{thm:main}
    The map
    \[
        {BSU_{\R}}_{2n}\to \left(\prod_{i=2}^{2n-1}K(\uZ,i\rho)\right) \times F_{2n}
    \]
    given by the Chern classes $c_i$ for $i<n$ and the Euler class $\epsilon$ is a rational $C_2$-equivalence.    
\end{thm}

\begin{cor}
    There are rational $C_2$-equivalences
    \[S^{2n\rho-1} \simeq \Sigma^{-1}F_{2n}\] 
    and 
    \[S^{(2n+1)\rho-1}\simeq K(\uZ,(2n+1)\rho-1)\times \mathrm{fib}\left(K(I,2n)\xra{(-)^2} K(I,4n) \right)\]
    where $(-)^2$ is the squaring map on fixed points.
\end{cor}

\begin{proof}
    This follows from the fiber sequence
    \[S^{m\rho-1}\to {BSU_{\R}}_{(m-1)}\to {BSU_{\R}}_m.\]
    The even case is clear. For the odd case $m=2n+1$, $S^{(2n+1)\rho-1}\simeq K(\uZ, (2n+1)\rho-1)\times H$ where $H$ is the fiber of the composite $F_{2n}\to K(A,2n\rho)\to K(\uZ,2n\rho)$. Since $K(I,2n\rho) = \mathrm{fib}\left(K(A,2n\rho)\to K(\uZ, 2n\rho)\right)$ and $F_{2n} = \mathrm{fib}\left(K(A,2n\rho)\to K(I,4n\rho)\right)$, $H$ is the fiber of the cup square $K(I,2n)\to K(I,4n)$. Here we have used the fact that $K(I,n\rho) = K(I,n)$.
\end{proof}

\bibliographystyle{alpha}
\bibliography{ref}

\end{document}